\documentclass[11pt]{article} 
\usepackage{amsmath,amssymb,amsthm}  
\usepackage{fancybox}  
\usepackage{graphicx} 
\usepackage{color}
\usepackage{colortbl}
\usepackage{mathdots}
\usepackage{lscape}

\allowdisplaybreaks
 
\begin{document}

\def\fl#1{\left\lfloor#1\right\rfloor}
\def\cl#1{\left\lceil#1\right\rceil}
\def\ang#1{\left\langle#1\right\rangle}
\def\stf#1#2{\left[#1\atop#2\right]} 
\def\sts#1#2{\left\{#1\atop#2\right\}}
\def\eul#1#2{\left\langle#1\atop#2\right\rangle}
\def\N{\mathbb N}
\def\Z{\mathbb Z}
\def\R{\mathbb R}
\def\C{\mathbb C}
\newcommand{\ctext}[1]{\raise0.2ex\hbox{\textcircled{\scriptsize{#1}}}}

\newtheorem{theorem}{Theorem}
\newtheorem{Prop}{Proposition}
\newtheorem{Cor}{Corollary}
\newtheorem{Lem}{Lemma}

\newenvironment{Rem}{\begin{trivlist} \item[\hskip \labelsep{\it
Remark.}]\setlength{\parindent}{0pt}}{\end{trivlist}}

\title{The Frobenius number for shifted geometric sequences associated with the number of solutions 
}

\author{
Takao Komatsu%\footnote{The current address: Faculty of Education, Nagasaki University, Nagasaki 852-8521, Japan} 
\\
\small Department of Mathematical Sciences, School of Science\\[-0.8ex]
\small Zhejiang Sci-Tech University\\[-0.8ex]
\small Hangzhou 310018 China\\[-0.8ex]
\small \texttt{komatsu@zstu.edu.cn}
}

\date{
%\small Submitted: February 15, 2021;  Accepted: March 25, 2021.\\
\small MR Subject Classifications: Primary 11D07; Secondary 05A15, 05A17, 05A19, 11B68, 11D04, 11P81 
}

\maketitle
 
\begin{abstract} 
For a non-negative integer $p$, one of the generalized Frobenius numbers, that is called the $p$-Frobenius number, is the largest integer that is represented at most in $p$ ways as a linear combination with nonnegative integer coefficients of a given set of positive integers whose greatest common divisor is one. The famous so-called Frobenius number proposed by Frobenius is reduced to the $0$-Frobenius number when $p=0$. 
The explicit formula for the Frobenius number with two variables was found in the 19th century, but a formula with more than two variables is very difficult to find, and closed formulas of Frobenius numbers have been found only in special cases such as geometric, Thabit, Mersenne, and so on. The case of $p>0$ was even more difficult, and not a single formula was known. 
However, most recently, we have finally succeeded in giving the $p$-Frobenius numbers as closed-form expressions of the triangular number triplet (\cite{Ko22a}), repunits (\cite{Ko22b}), Fibonacci triplet (\cite{Ko23b}) and Jacobsthal triplet (\cite{KP}).

In this paper, 
we give closed-form expressions of the $p$-Frobenius number for the finite sequence $\{a b^n-c\}_n$, where $a$, $b$ and $c$ are integers with $a\ge 1$, $b\ge 2$ and $c\ne 0$. This sequence includes the cases for geometric, Thabit and Mersenne as well as their variations.
\end{abstract}

\section{Introduction}  

For given positive integers $a_1,\dots,a_k$ with $\gcd(a_1,\dots,a_k)=1$, the largest positive integer with no nonnegative integer representation by $a_1,\dots,a_k$ is called the {\it Frobenius number} and denoted by $g=g(a_1,\dots,a_k)$. In numerical semigroups, $g+1$ is called the {\it conductor}. Such a problem is also known as the coin exchange problem. Many methods have been used in intending to find a formula for the Frobenius number. The Frobenius problem is easy to solve when $n=2$. However, the computation of an explicit formula when $n=3$ is much more difficult and has been studied by many researchers for a long time. 

Another interesting matter is the number of nonnegative integer representations of a given nonnegative integer $m$ by $a_1,\dots,a_k$, which is called {\it denumerant}. The notion of a denumerant is an important generalization of the notion of a number of partitions (\cite{Caley}). In \cite{ko03}, a general form that is well computable practically to find the denumerant is given. This has been improved by Binner \cite{bi21} recently.  

Therefore, in a natural generalization, for a nonnegative integer $p$, consider the set $S_p$ of all nonnegative integers whose numbers of representations with nonnegative integers in terms of $a_1,\dots,a_k$ are more than $p$. When $\gcd(a_1,\dots,a_k)=1$, there exists the largest integer in $\mathbb N_0\backslash S_p$. Namely, the largest positive integer whose number of representations with nonnegative integers in terms of $a_1,\dots,a_k$ is at most $p$. That is, all integers larger than this number can be represented by $a_1,\dots,a_k$ in $p+1$ or more ways. Such a general Frobenius number may be called the $p$-Frobenius number, and denoted by $g_p(a_1,\dots,a_k)$. So, when $p=0$, $g(a_1,\dots,a_k)=g_0(a_1,\dots,a_k)$ is the original Frobenius number.  

In \cite{op08,tr08}, the Frobenius number for geometric sequences of the form $a^k,a^{k-1}b,\dots,a b^{k-1},b^k$ has been determined. Note that the familiar geometric sequence of the form $a,a r,a r^2,\dots,a r^{k-1}$ does not satisfy the condition $\gcd(A)=1$ if $a$ and $r$ are positive integers with $a,r\ge 2$. In this paper, we are interested in another way to satisfy this condition by shifting by an integer $c$.  

Let $S=\ang{a_1,\dots,a_k}$ be the set of all nonnegative integral linear combinations of $a_1,\dots,a_k$. Then $S$ is a numerical semigroup. In this sense, the set $S_p$ is called the {\it $p$-numerical semigroup} with $S=S_0$.   
In \cite{Song20}\footnote{Example 5.14 (4) is wrong.}, Thabit numerical semigroups of the first kind in base $b$ and of the second kind in base $b$, defined as $\ang{\{b^{n+i}(b+1)-1|i\in\mathbb N_0\}}$ and $\ang{\{b^{n+i}(b+1)+1|i\in\mathbb N_0\}}$, respectively, and Cunningham numerical semigroups defined as $\ang{\{b^{n+i}+1|i\in\mathbb N_0\}}$ are studied. Here, $\mathbb N_0$ denotes the set of nonnegative integers. When $b=2$ in Cunningham, the Frobenius number is obtained in \cite{Shallit21}. 
In \cite{Song21}, generalized Thabit numerical semigroups, defined as $\ang{\{2^{n+i}(2^k+1)-(2^k-1)|i\in\mathbb N_0\}}$, are studied to obtain the Frobenius number and so on. When $k=1$, Mersenne numerical semigroups are studied in \cite{RBT17}. When $b=2$ in Thabit numerical semigroups of the first kind or $k=1$ in generalized Thabit, Thabit numerical semigroups are studied in \cite{RBT15}. Another similar type is on repunit numerical semigroups $\ang{\{(b^{n+i}-1)/(b-1)|i\in\mathbb N_0\}}$ (\cite{RBT16}) and more general types are considered in \cite{AB}. 

In this paper, for positive integers $a$, $b$ and $c$ with $b\ge 2$, we consider a more general type of semigroup $\ang{\{a b^{n+i}-c|i\in\mathbb N_0\}}$. When $p=0$, our concern is that there are representations or are not representations, that is, the number of representations is either $0$, or $1$ or more, so it has no effect whether the sequence is infinite or finite. However, in this paper, we are interested in whether the number of representations is less than or greater than $p$ for the general nonnegative integer $p$. Therefore, the situation is different when the sequence is infinite and when it is finite. Hence, in this paper, consider only the case of a {\it finite} sequence\footnote{Since the considered sequence is infinite in \cite{AB,RBT15,RBT16,RBT17,Shallit21,Song20,Song21}, the result can be different even when $p=0$.} sequence starting with $a b^{n}-c, a b^{n+1}-c, \dots$.     
For two variables, the results are easy to know. For general positive integers $a$ and $b$ exceeding $1$ with $\gcd(a,b)=1$, 
$$
g_p(a,b)=(p+1)a b-a-b 
$$ 
(see, e.g., \cite{bk11,Ko-p}). However, when there are three or more variables, it is sufficiently difficult to obtain the explicit formula of $g_p$ or $n_p$ even in a special case when $p=0$, and it is even more difficult when $p>0$, and in the case of a special triplet even one example had not been found. However, very recently, in \cite{Ko22a}, \cite{Ko22b}, \cite{Ko23b} and \cite{KP}, for $p\ge 0$ we finally succeeded in obtaining explicit formulas in the sequences of triangular numbers, repunits and Fibonacci numbers, respectively.  
Therefore, our main concern is in the case of three variables, that is  
$$
\{a b^{n}-c,~ a b^{n+1}-c,~ a b^{n+2}-c\}
$$ 
with $\gcd(a b^{n}-c,~ a b^{n+1}-c,~ a b^{n+2}-c)=1$, 
which is the main topic in this paper.  

The nonnegative integers not in $S$ are called the gaps of $S$. We shall also see the number of gaps, which is called the {\it genus} in numerical semigroup or the {\it Sylvester number}. The existence or composition of one-point algebraic curves for numerical semigroups when the genus is small is one of the most popular and interesting topics (see, e.g., \cite{KaKo}) We also find the generalized Sylvester number or the $p$-Sylvester number for $S=\{a b^{3}-c, a b^{4}-c, a b^{5}-c\}$.

\section{Ap\'ery set}  

We introduce the Ap\'ery set \cite{Apery} to obtain the formulas in this paper. 
 
Let $p$ be a nonnegative integer. For a set of positive integers $A=\{a_1,a_2,\dots,a_k\}$ with $\gcd(A)=1$ and $a_1=\min(A)$ we denote by 
$$
{\rm Ap}_p(A)={\rm Ap}(a_1,a_2,\dots,a_k)=\{m_0^{(p)},m_1^{(p)},\dots,m_{a_1-1}^{(p)}\}\,, 
$$ 
the Ap\'ery set of $A$, where each positive integer $m_i^{(p)}$ $(0\le i\le a_1-1)$ satisfies the conditions:
$$
{\rm (i)}\, m_i^{(p)}\equiv i\pmod{a_1},\quad{\rm (ii)}\, m_i^{(p)}\in S_p(A),\quad{\rm (iii)}\, m_i^{(p)}-a_1\not\in S_p(A)
$$ 
Note that $m_0$ is defined to be $0$.

It follows that for each $p$, 
$$
{\rm Ap}_p(A)\equiv\{0,1,\dots,a_1-1\}\pmod{a_1}\,. 
$$  

One of the convenient formulas to obtain the $p$-Frobenius number is via the elements in the corresponding Ap\'ery set (\cite{Ko-p}).  

\begin{Lem}  
Let $\gcd(a_1,\dots,a_k)=1$ with $a_1=\min\{a_1,\dots,a_k\}$. Then we have 
$$
g_p(a_1,\dots,a_k)=\max_{0\le j\le a_1-1}m_j^{(p)}-a_1\,. 
$$ 
\label{lem:p-frob}
\end{Lem}  

\noindent 
{\it Remark.}  
When $p=0$, it is essentially due to Brauer and Shockley \cite{bs62}.  
More general formulas, including the $p$-power sum and the $p$-weighted sum, can be seen in \cite{Ko-p}.

\section{Main result} 

For a given positive integer $n$, 
let $q$ and $r$ be nonnegative integers satisfying $a b^n-c=(b+1)q+r$ with $0\le r<b+1$. That is, 
$$
q=\fl{\frac{a b^n-c}{b+1}}\quad\hbox{and}\quad r=(a b^n-c)-(b+1)\fl{\frac{a b^n-c}{b+1}}\,. 
$$ 
For simplicity, put $A_{3,n}=(a b^n-c,a b^{n+1}-c, a b^{n+2}-c)$.  In this paper, we study the generalized Frobenius number ($p$-Frobenius number) for the triple $A_{3,n}$. In this section, we consider the case $c>0$. When $c<0$ and $r>1$, it is a little more difficult to decide the largest element in ${\rm Ap}_p(A_{3,n})$. We discuss the case $c<0$ in the later section.  

\begin{theorem} 
If $\gcd(A_{3,n})=1$, then for $0\le p\le q$ we have 
$$
g_p(A_{3,n})
=\left\{
\begin{alignedat}{2}
& (r-1)(a b^{n+1}-c)+(q+p)(a b^{n+2}-c)-(a b^n-c)\\
&\qquad\qquad\qquad\qquad\qquad\text{if $(b+1)\nmid(a b^n-c)$ and $c>0$},\\
& b(a b^{n+1}-c)+(q+p-1)(a b^{n+2}-c)-(a b^n-c)\\
&\qquad\qquad\qquad\qquad\qquad\text{if $(b+1)\mid(a b^n-c)$}\,.
\end{alignedat}
\right. 
$$ 
\label{th:p-frob} 
\end{theorem}

First, suppose that $r\ge 1$, that is $(b+1)\nmid(a b^n-c)$.   
Table \ref{tb:g0system} shows an array of elements in ${\rm Ap}_0(A_{3,n})$, where $c'=(b-1)c$ for short. Notice that 
\begin{multline*} 
\{0,1,2,\dots,a b^n-c-1\}\\
=\{0,(b-1)c,2(b-1)c,\dots,(q(b+1)+r-1)(b-1)c\}\pmod{q(b+1)+r}
\end{multline*}  
because $\gcd(A_{3,n})=1$ and $\gcd(b-1,c)=1$.  
Since $a b^{n+1}-c\equiv(b-1)c=c'\pmod{a b^n-c}$ and $a b^{n+2}-c\equiv(b^2-1)c=(b+1)c'\pmod{a b^n-c}$, for $0\le x_2\le b$ and $0\le x_3\le q-1$ or $0\le x_2\le r-1$ and $x_3=q$, we have 
$$ 
(a b^{n+1}-c)x_2+(a b^{n+2}-c)x_3\equiv c' x_2+(b+1)x_3\pmod{a b^n-c}\,. 
$$ 

In Table \ref{tb:g012psystem}, $(x_2,x_3)$ denotes the corresponding element $(a b^{n+1}-c)x_2+(a b^{n+2}-c)x_3$. 

As shown in Table \ref{tb:g01system}, based on the array of elements in ${\rm Ap}_0(A_{3,n})$, the corresponding array of elements in ${\rm Ap}_1(A_{3,n})$ is determined. In detail, elements that have the same modulo as the elements in the columns from the second row to the last row appear as they are by moving up one row to the right block. Only the elements of the same modulo that correspond to the $b+1$ elements in the top row are arranged to continue from the final element in ${\rm Ap}_0(A_{3,n})$, following the bottom row of the same block. 
In the main part, there exists the correspondence $(x_2,x_3)\Longrightarrow(x_2+b+1,x_3-1)$ because 
\begin{align*}
&(a b^{n+1}-c)x_2+(a b^{n+2}-c)x_3\\
&\equiv(a b^{n+1}-c)(x_2+b+1)+(a b^{n+2}-c)(x_3-1)\pmod{a b^n-c}\,. 
\end{align*}    
Concerning the elements in the top row, there are the correspondences $(x_2,0)\Longrightarrow(x_2+r,q)$ ($0\le x_2\le b-r$) and $(x_2,0)\Longrightarrow(x_2-b+r-1,q+1)$ ($b-r+1\le x_2\le b$) because  
\begin{align*}
(a b^{n+1}-c)x_2&\equiv(a b^{n+1}-c)(x_2+r)+(a b^{n+2}-c)q\,,\\ 
(a b^{n+1}-c)x_2&\equiv(a b^{n+1}-c)(x_2-b+r-1)+(a b^{n+2}-c)(q+1)\\
&\qquad\qquad \pmod{a b^n-c}\,, 
\end{align*}
respectively.  

We can show that each element in ${\rm Ap}_1(A_{3,n})$ is expressed in exactly two ways. For simplicity, write $(x_1,x_2,x_2):=(a b^n-c)x_1+(a b^{n+1}-c)x_2+(a b^{n+2}-c)x_3$. 
Concerning the main part in the second block, by 
\begin{equation}
(b+1)(a b^{n+1}-c)=b(a b^n-c)+(a b^{n+2}-c)\,, 
\label{eq:435} 
\end{equation}
we have 
$$
(0,b+x_2,x_3)=(b,x_2-1,x_3+1)\quad(1\le x_2\le b+1;\, x_3\ge 0)\,. 
$$ 
Concerning the ($q+1$)st elements in the first block, by the definition of $q$ and $r$, we have 
$$
(0,r+x_2,q)=(a b^{n+1}-c-b q,x_2,q)\quad (x_2\ge 0)\,. 
$$ 
Concerning the ($q+2$)nd elements in the first block, by the equation (\ref{eq:435}), we have 
$$
(0,x_2,q+1)=(a b^{n+1}-c-b(q+1),b+1-r+x_2,0)\quad (x_2\ge 0)\,. 
$$ 
However, each value minus $(a b^n-c)$ does not have two expressions but only one because, for example, $(-1,b+x_2,x_3)$ does not make sense.  

The same rule applies to ${\rm Ap}_2(A_{3,n})$, ${\rm Ap}_2(A_{3,n})$, and until ${\rm Ap}_p(A_{3,n})$.  
In Table \ref{tb:g01psystem}, $\ctext{j}$ ($j=0,1,2,\dots,q$) shows the range in which all the elements in ${\rm Ap}_j(A_{3,n})$ exist. For short, $q'=q-1$. 

And eventually, each element in ${\rm Ap}_q(A_{3,n})$ is expressed in exactly $q+1$ ways, and each element minus $(a b^n-c)$ in exactly $q$ ways. For example, concerning the elements in the bottom row of the leftmost block, by (\ref{eq:435}) we have 
\begin{align*}  
&(0,x_2,2 q)=(a b^{n+1}-c-b(q+1),(b+1)-r+x_2,q-1)\\
&=(a b^{n+1}-c-b(q+2),2(b+1)-r+x_2,q-2)\\
&=(a b^{n+1}-c-b(q+3),3(b+1)-r+x_2,q-3)\\
&=\cdots=(a b^{n+1}-c-b\cdot 2 q,q(b+1)-r+x_2,0)\quad(0\le x_2\le r-1)\,. 
\end{align*}

\begin{table}[htbp]
  \centering
\scalebox{0.8}{
\begin{tabular}{ccccc}
\cline{1-2}\cline{3-4}\cline{5-5}
\multicolumn{1}{|c}{$0$}&$c'$&$2 c'$&$\cdots$&\multicolumn{1}{c|}{$b c'$}\\
\multicolumn{1}{|c}{$(b+1)c'$}&$(b+2)c'$&$(b+3)c'$&$\cdots$&\multicolumn{1}{c|}{$(2 b+1)c'$}\\
\multicolumn{1}{|c}{$(2 b+2)c'$}&&&&\multicolumn{1}{c|}{}\\
\multicolumn{1}{|c}{$\vdots$}&&&&\multicolumn{1}{c|}{}\\
\multicolumn{1}{|c}{$(q-1)(b+1)c'$}&&&&\multicolumn{1}{c|}{$(q(b+1)-1)c'$}\\
\cline{4-5}
\multicolumn{1}{|c}{$q(b+1)c'$}&$\cdots$&\multicolumn{1}{c|}{$(q(b+1)+r-1)c'$}&&\\
\cline{1-2}\cline{3-3}
\end{tabular}
} 
  \caption{Complete residue system ${\rm Ap}_0(A_{3,n})$}
  \label{tb:g0system}
\end{table}

\begin{table}[htbp]
  \centering
\scalebox{0.5}{
\begin{tabular}{cccccccccc}
\multicolumn{1}{|c}{}&&&&&\multicolumn{1}{|c}{}&&&&\multicolumn{1}{c|}{}\\
\cline{1-2}\cline{3-4}\cline{5-6}\cline{7-8}\cline{9-10}
\multicolumn{1}{|c}{$0$}&$c'$&$\cdots$&$\cdots$&$b c'$&\multicolumn{1}{|c}{$(b+1)c'$}&$(b+2)c'$&$\cdots$&$\cdots$&\multicolumn{1}{c|}{$(2 b+1)c'$}\\
\multicolumn{1}{|c}{$(b+1)c'$}&$(b+2)c'$&$\cdots$&$\cdots$&$(2 b+1)c'$&\multicolumn{1}{|c}{$(2 b+2)c'$}&$(2 b+3)c'$&&&\multicolumn{1}{c|}{$(3 b+2)c'$}\\
\multicolumn{1}{|c}{$(2 b+2)c'$}&$(2 b+3)c'$&&&$(3 b+2)c'$&\multicolumn{1}{|c}{$(3 b+3)c'$}&&&&\multicolumn{1}{c|}{}\\
\multicolumn{1}{|c}{$\vdots$}&&&&&\multicolumn{1}{|c}{$\vdots$}&&&&\multicolumn{1}{c|}{}\\
\multicolumn{1}{|c}{$\vdots$}&&&&&\multicolumn{1}{|c}{$(q-1)(b+1)c'$}&&&&\multicolumn{1}{c|}{$(q(b+1)-1)c'$}\\
\cline{9-10}
\multicolumn{1}{|c}{$(q-1)(b+1)c'$}&$\cdots$&&&\multicolumn{1}{c|}{$(q(b+1)-1)c'$}&$q(b+1)c'$&$\cdots$&\multicolumn{1}{c|}{$(q(b+1)+r-1)c'$}&&\\
\cline{4-4}\cline{5-6}\cline{7-8}
\multicolumn{1}{|c}{$q(b+1)c'$}&$\cdots$&\multicolumn{1}{c|}{$(q(b+1)+r-1)c'$}&$0\cdots$&\multicolumn{1}{c|}{$(b-r)c'$}&&&&&\\
\cline{1-2}\cline{3-4}\cline{5-5}
\multicolumn{1}{|c}{$(b-r+1)c'$}&$\cdots$&\multicolumn{1}{c|}{$b c'$}&&&&&&&\\
\cline{1-2}\cline{3-3}
\end{tabular}
} 
  \caption{Complete residue systems ${\rm Ap}_0(A_{3,n})$ and ${\rm Ap}_1(A_{3,n})$}
  \label{tb:g01system}
\end{table}

\begin{table}[htbp]
  \centering
\scalebox{0.8}{
\begin{tabular}{ccccccccccccccccc}
\cline{1-2}\cline{3-4}\cline{5-6}\cline{7-8}\cline{9-9}\cline{13-14}\cline{15-16}\cline{17-17}
\multicolumn{1}{|c}{}&&&\multicolumn{1}{|c}{}&&&\multicolumn{1}{|c}{}&&&\multicolumn{1}{|c}{$\cdots$}&$\cdots$&$\cdots$&\multicolumn{1}{|c}{}&$\ctext{q'}$&&\multicolumn{1}{|c}{$\ctext{q}$}&\multicolumn{1}{c|}{}\\ 
\cline{15-16}\cline{17-17}
\multicolumn{1}{|c}{}&&&\multicolumn{1}{|c}{}&&&\multicolumn{1}{|c}{}&&&\multicolumn{1}{|c}{}&&&\multicolumn{1}{|c}{}&\multicolumn{1}{c|}{}&\multicolumn{1}{c|}{$\ctext{q}$}&&\\ 
\cline{13-14}\cline{15-15}
\multicolumn{1}{|c}{}&&&\multicolumn{1}{|c}{}&$\ctext{1}$&&\multicolumn{1}{|c}{}&$\ctext{2}$&&\multicolumn{1}{|c}{}&&\multicolumn{1}{c|}{}&&\multicolumn{1}{c|}{$\ctext{q}$}&&&\\ 
\cline{13-14}
\multicolumn{1}{|c}{}&$\ctext{0}$&&\multicolumn{1}{|c}{}&&&\multicolumn{1}{|c}{}&&&\multicolumn{1}{|c}{}&&&&&&&\\ 
\cline{9-9}
\multicolumn{1}{|c}{}&&&\multicolumn{1}{|c}{}&&&\multicolumn{1}{|c}{}&\multicolumn{1}{c|}{}&&&&&&&&&\\ 
\cline{6-6}\cline{7-8}
\multicolumn{1}{|c}{}&&&\multicolumn{1}{|c}{}&&\multicolumn{1}{|c}{$\ctext{2}$}&\multicolumn{1}{|c}{}&&&&&&&&&&\\ 
\cline{3-4}\cline{5-6}
\multicolumn{1}{|c}{}&&\multicolumn{1}{|c}{$\ctext{1}$}&\multicolumn{1}{|c}{}&$\ctext{2}$&\multicolumn{1}{|c}{}&&&&&&&&&&&\\ 
\cline{1-2}\cline{3-4}\cline{5-5}\cline{9-9}
\multicolumn{1}{|c}{}&$\ctext{1}$&\multicolumn{1}{|c}{$\ctext{2}$}&\multicolumn{1}{|c}{}&&&&&\multicolumn{1}{|c}{$\ctext{q'}$}&\multicolumn{1}{|c}{}&&&&&&&\\ 
\cline{1-2}\cline{3-3}\cline{7-8}\cline{9-9}
\multicolumn{1}{|c}{}&$\ctext{2}$&\multicolumn{1}{|c}{}&&&&\multicolumn{1}{|c}{}&$\ctext{q'}$&$\ctext{q}$&\multicolumn{1}{|c}{}&&&&&&&\\ 
\cline{1-2}\cline{6-7}\cline{8-9}
&$\vdots$&&&&\multicolumn{1}{|c}{$\ctext{q'}$}&\multicolumn{1}{|c}{}&$\ctext{q}$&\multicolumn{1}{|c}{}&&&&&&&&\\ 
\cline{4-4}\cline{5-6}\cline{7-8}
&$\vdots$&&\multicolumn{1}{|c}{}&$\ctext{q'}$&\multicolumn{1}{|c}{$\ctext{q}$}&\multicolumn{1}{|c}{}&&&&&&&&&&\\ 
\cline{3-4}\cline{5-6}
&$\vdots$&\multicolumn{1}{|c}{$\ctext{q'}$}&\multicolumn{1}{|c}{}&$\ctext{q}$&\multicolumn{1}{|c}{}&&&&&&&&&&&\\ 
\cline{1-2}\cline{3-4}\cline{5-5}
\multicolumn{1}{|c}{}&$\ctext{q'}$&\multicolumn{1}{|c}{$\ctext{q}$}&\multicolumn{1}{|c}{}&&&&&&&&&&&&&\\ 
\cline{1-2}\cline{3-3}
\multicolumn{1}{|c}{}&$\ctext{q}$&\multicolumn{1}{|c}{}&&&&&&&&&&&&&&\\ 
\cline{1-2}
\end{tabular}
} 
  \caption{Complete residue systems ${\rm Ap}_0(A_{3,n})$ until ${\rm Ap}_p(A_{3,n})$}
  \label{tb:g01psystem}
\end{table}

\begin{landscape}
\begin{table}[htbp]
  \centering
\scalebox{0.65}{
\begin{tabular}{cccccccccccccccc}
\multicolumn{1}{|c}{}&&&&&\multicolumn{1}{|c}{}&&&&&\multicolumn{1}{|c}{}&&&&&\multicolumn{1}{|c}{}\\ 
\cline{1-2}\cline{3-4}\cline{5-6}\cline{7-8}\cline{9-10}\cline{11-12}\cline{13-14}\cline{15-15}
\multicolumn{1}{|c}{$(0,0)$}&$(1,0)$&$\cdots$&$\cdots$&$(b,0)$&\multicolumn{1}{|c}{$(b+1,0)$}&$(b+2,0)$&$\cdots$&$\cdots$&$(2 b+1,0)$&\multicolumn{1}{|c}{$(2 b+2,0)$}&&&&$(3 b+2,0)$&\multicolumn{1}{|c}{$\cdots$}\\
\multicolumn{1}{|c}{$(0,1)$}&$(1,1)$&$\cdots$&$\cdots$&$(b,1)$&\multicolumn{1}{|c}{$(b+1,1)$}&$(b+2,1)$&$\cdots$&$\cdots$&$(2 b+1,1)$&\multicolumn{1}{|c}{$(2 b+2,1)$}&&&&$(3 b+2,1)$&\multicolumn{1}{|c}{}\\
\multicolumn{1}{|c}{}&&&&$\vdots$&\multicolumn{1}{|c}{$\vdots$}&&&&$\vdots$&\multicolumn{1}{|c}{$\vdots$}&&&&$\vdots$&\multicolumn{1}{|c}{$\vdots$}\\
\multicolumn{1}{|c}{$\vdots$}&&&&$\vdots$&\multicolumn{1}{|c}{$\vdots$}&&&&$\vdots$&\multicolumn{1}{|c}{$(2 b+2,q-3)$}&$\cdots$&&$\cdots$&$(3 b+2,q-3)$&\multicolumn{1}{|c}{}\\
\cline{14-15}
\multicolumn{1}{|c}{$(0,q-2)$}&$(1,q-2)$&$\dots$&$\dots$&$(b,q-2)$&\multicolumn{1}{|c}{$(b+1,q-2)$}&$\cdots$&$\dots$&$\dots$&$(2 b+1,q-2)$&\multicolumn{1}{|c}{$(2 b+2,q-2)$}&&\multicolumn{1}{c|}{$(2 b+r+1,q-2)$}&&&\\ 
\cline{9-10}\cline{11-12}\cline{13-13}
\multicolumn{1}{|c}{$(0,q-1)$}&$(1,q-1)$&$\dots$&$\dots$&$(b,q-1)$&\multicolumn{1}{|c}{$(b+1,q-1)$}&$\dots$&\multicolumn{1}{c|}{$(b+r,q-1)$}&$\cdots$&\multicolumn{1}{c|}{$(2 b+1,q-1)$}&&&&&&\\ 
\cline{4-4}\cline{5-6}\cline{7-8}\cline{9-10}
\multicolumn{1}{|c}{$(0,q)$}&$\dots$&\multicolumn{1}{c|}{$(r-1,q)$}&$(r,q)\dots$&\multicolumn{1}{c|}{$(b,q)$}&$(b+1,q)$&$\cdots$&\multicolumn{1}{c|}{$(b+r,q)$}&&&&&&&&\\ 
\cline{1-2}\cline{3-4}\cline{5-6}\cline{7-8} 
\multicolumn{1}{|c}{$(0,q+1)$}&$\dots$&\multicolumn{1}{c|}{$(r-1,q+1)$}&$(r,q+1)\dots$&\multicolumn{1}{c|}{$(b,q+1)$}&&&&&&&&&&&\\ 
\cline{1-2}\cline{3-4}\cline{5-5} 
\multicolumn{1}{|c}{$(0,q+2)$}&$\dots$&\multicolumn{1}{c|}{$(r-1,q+2)$}&$\dots$&\multicolumn{1}{c|}{$\dots$}&&&&&&&&&&&\\ 
\cline{1-2}\cline{3-4}\cline{5-5} 
$\vdots$&&&&$\vdots$&&&&&&&&&&\\ 
$\vdots$&&&&$\vdots$&&&&&&&&&&\\ 
\cline{1-2}\cline{3-4}\cline{5-5} 
\multicolumn{1}{|c}{$(0,2 q-1)$}&$\dots$&\multicolumn{1}{c|}{$(r-1,2 q-1)$}&$(r,2 q-1)\dots$&\multicolumn{1}{c|}{$(b,2 q-1)$}&&&&&&&&&&&\\ 
\cline{1-2}\cline{3-4}\cline{5-5} 
\multicolumn{1}{|c}{$(0,2 q)$}&$\dots$&\multicolumn{1}{c|}{$(r-1,2 q)$}&&&&&&&&&&&&&\\ 
\cline{1-2}\cline{3-3}
\end{tabular}
}  
\vskip1cm 
\scalebox{0.65}{ 
\begin{tabular}{cccccccccccc}
&&&&\multicolumn{1}{|c}{}&&&&&\multicolumn{1}{|c}{}&&\multicolumn{1}{c|}{}\\
\cline{5-6}\cline{7-8}\cline{9-10}\cline{11-12}
$\dots$&$\dots$&$\dots$&$\dots$&\multicolumn{1}{|c}{$((q-1)(b+1),0)$}&$\dots$&$\dots$&$\dots$&\multicolumn{1}{c|}{$(q(b+1)-1,0)$}&$(q(b+1),0)$&$\dots$&\multicolumn{1}{c|}{$(q(b+1)+r-1,0)$}\\
\cline{5-6}\cline{7-8}\cline{9-10}\cline{11-12}
&&&&\multicolumn{1}{|c}{$((q-1)(b+1),1)$}&$\dots$&\multicolumn{1}{c|}{$((q-1)(b+1)+r-1,1)$}&$\dots$&\multicolumn{1}{c|}{$(q(b+1)-1,1)$}&&&\\
\cline{5-6}\cline{7-8}\cline{9-9}
&&&&\multicolumn{1}{|c}{$((q-1)(b+1),2)$}&$\dots$&\multicolumn{1}{c|}{$((q-1)(b+1)+r-1,2)$}&&&&&\\
\cline{5-6}\cline{7-7}
\end{tabular} 
} 
  \caption{Complete residue systems ${\rm Ap}_0(A_{3,n})$, ${\rm Ap}_1(A_{3,n})$, ${\rm Ap}_2(A_{3,n})$, $\dots$, ${\rm Ap}_p(A_{3,n})$}
  \label{tb:g012psystem}
\end{table} 
\end{landscape}

In ${\rm Ap}_0(A_{3,n})$, the largest element is at $(r-1,q)$ or at $(b,q-1)$. Since by $r\ge 1$ and $c>0$ 
\begin{equation}
(a b^{n+1}-c)(r-1)+(a b^{n+2}-c)q>(a b^{n+1}-c)b+(a b^{n+2}-c)(q-1) 
\label{eq:ape0,c>0}
\end{equation} 
and Lemma \ref{lem:p-frob}, we have 
$$
g_0(A_{3,n})=(a b^{n+1}-c)(r-1)+(a b^{n+2}-c)q-(a b^n-c)\,. 
$$ 
Note that if $c<0$, the relation (\ref{eq:ape0,c>0}) is not necessarily valid. We discuss the case $c<0$ in the later section.  
In ${\rm Ap}_1(A_{3,n})$, comparing the elements at $(r-1,q+1)$, at $(b,q)$ or at $(b+r,q-1)$, by $c>0$, we have 
$$
g_1(A_{3,n})=(a b^{n+1}-c)(r-1)+(a b^{n+2}-c)(q+1)-(a b^n-c)\,. 
$$ 
Similarly, for $0\le p\le q$ the largest element exists at $(r-1,q+p)$. Hence, we have 
$$
g_p(A_{3,n})=(a b^{n+1}-c)(r-1)+(a b^{n+2}-c)(q+p)-(a b^n-c)\,. 
$$ 
However, when $p=q+1$, no element comes to the position at $(r-1,q+p+1)$ because there are only $r$ elements at the top row in the rightmost block of ${\rm Ap}_p(A_{3,n})$.

Next, suppose that $r=0$, that is $(b+1)\mid(a b^n-c)$. In this case, there is no ($q+1$)-st row in Table \ref{tb:g0system}, and eventually, the rightmost block constituted from $r$ elements in Table \ref{tb:g012psystem} does not exist. Hence, the maximal element in ${\rm Ap}_p$ is at the position $(b,q+p-1)$ ($0\le p\le q$). Therefore,   
$$
g_p(A_{3,n})=(a b^{n+1}-c)b+(a b^{n+2}-c)(q+p-1)-(a b^n-c)\,. 
$$

\section{Sylvester number}  

The number of positive integers with no nonnegative integer representation by $a_1,\dots,a_k$ is called the {\it Sylvester number} and denoted by $n(a_1,\dots,a_k)$. In numerical semigroups, it is often called the {\it genus}. The genus has a very important role in the study of algebraic curves (see, e.g., \cite{KaKo}). In a natural generalization, consider the number of positive integers whose number of representations with nonnegative integers in terms of $a_1,\dots,a_k$ is at most $p$. Such a general Sylvester number may be called the $p$-Sylvester number or the $p$-genus, and denoted by $n_p(a_1,\dots,a_k)$. So, when $p=0$, $n(a_1,\dots,a_k)=n_0(a_1,\dots,a_k)$ is the original Sylvester number.  
One of the convenient formulas to obtain the $p$-Sylvester number is via the elements in the corresponding Ap\'ery set (\cite{Ko-p}).  

\begin{Lem}  
Let $\gcd(a_1,\dots,a_k)=1$ with $a_1=\min\{a_1,\dots,a_k\}$. Then we have 
$$
n_p(a_1,\dots,a_k)=\frac{1}{a_1}\sum_{j=0}^{a_1-1}m_j^{(p)}-\frac{a_1-1}{2}\,. 
$$ 
\label{lem:p-sylv}
\end{Lem}  

\noindent 
{\it Remark.}  
When $p=0$, it is essentially due to Selmer \cite{se77}.

\begin{theorem}  
If $\gcd(A_{3,n})=1$, then for $0\le p\le q$ we have 
\begin{align*} 
n_p(A_{3,n})&=\frac{1}{2}\bigl((a b^n-c-1)(a b^{n+1}-c-1)-b q(2(a b^n-c)-(b+1)(q+1))\\
&\qquad +(b+1)p(2(a b^{n+1}-c)-b(p+1))\bigr)\,. 
\end{align*} 
\label{th:p-sylv}
\end{theorem}  
\begin{proof}  
We can find the elements in ${\rm Ap}_p(A_{3,n})$ as follows.  Regarding the rightmost main block, the elements exist at 
\begin{align*}  
&\bigl(p(b+1),0\bigr),~\bigl(p(b+1)+1,0\bigr),~\dots,~\bigl((p+1)(b+1)-1,0\bigr),\\
&\bigl(p(b+1),1\bigr),~\bigl(p(b+1)+1,1\bigr),~\dots,~\bigl((p+1)(b+1)-1,1\bigr),\\
&\qquad\qquad\qquad\cdots\\
&\bigl(p(b+1),q-p-1\bigr),~\bigl(p(b+1)+1,q-p-1\bigr),\\
&\qquad\qquad\qquad\qquad\qquad\qquad~\dots,~\bigl((p+1)(b+1)-1,q-p-1\bigr),\\
&\bigl(p(b+1),q-p\bigr),~\bigl(p(b+1)+1,q-p\bigr),~\dots,~\bigl((p+1)(b+1)+r-1,q-p\bigr)\,. 
\end{align*} 
Regarding the parts in the remaining stairs, the elements exist at  
\begin{align*}  
&\qquad\qquad \bigl((p-1)(b+1)+r,q-p+1\bigr),~\dots,~\bigl(p(b+1)-1,q-p+1\bigr),\\
&
\bigl((p-1)(b+1),q-p+2\bigr),~\dots,~\bigl((p-1)(b+1)+r-1,q-p+2\bigr),\\ 
&\qquad\qquad \bigl((p-2)(b+1)+r,q-p+3\bigr),~\dots,~\bigl((p-1)(b+1)-1,q-p+3\bigr),\\
&
\bigl((p-2)(b+1),q-p+4\bigr),~\dots,~\bigl((p-2)(b+1)+r-1,q-p+4\bigr),\\ 
&\qquad\qquad\qquad\cdots\\ 
&\qquad\qquad (b+1+r,q+p-3\bigr),~\dots,~\bigl(2(b+1)-1,q+p-3\bigr),\\
&
(b+1,q+p-2\bigr),~\dots,~(b+r,q+p-2),\\ 
&\qquad\qquad (r,q+p-1\bigr),~\dots,~(b,q+p-1),\\
&
(0,q+p\bigr),~\dots,~(r-1,q+p)\,.  
\end{align*} 
Hence, 
\begin{align*} 
&\sum_{j=0}^{a_1-1}m_j^{(p)}\\
&=(q-p)\left(p(b+1)^2+\sum_{i=0}^b i\right)(a b^{n+1}-c)+(b+1)\left(\sum_{i=0}^{q-p-1}j\right)(a b^{n+2}-c)\\
&\quad +\left(p(b+1)r+\sum_{i=0}^{r-1}i\right)(a b^{n+1}-c)+r(q-p)(a b^{n+2}-c)\\
&\quad +\left((p-1)(b+1)(b-r+1)+\sum_{i=r}^b i\right)(a b^{n+1}-c)\\
&\qquad +(b-r+1)(q-p+1)(a b^{n+2}-c)\\ 
&\quad +\left((p-1)(b+1)r+\sum_{i=0}^{r-1}i\right)(a b^{n+1}-c)+r(q-p+2)(a b^{n+2}-c)\\
&\quad +\left((p-2)(b+1)(b-r+1)+\sum_{i=r}^b i\right)(a b^{n+1}-c)\\
&\qquad +(b-r+1)(q-p+3)(a b^{n+2}-c)\\ 
&\quad +\left((p-2)(b+1)r+\sum_{i=0}^{r-1}i\right)(a b^{n+1}-c)+r(q-p+4)(a b^{n+2}-c)\\ &\quad +\cdots\\
&\quad +\left((b+1)(b-r+1)+\sum_{i=r}^b i\right)(a b^{n+1}-c)+(b-r+1)(q+p-3)(a b^{n+2}-c)\\ 
&\quad +\left((b+1)r+\sum_{i=0}^{r-1}i\right)(a b^{n+1}-c)+r(q+p-2)(a b^{n+2}-c)\\ 
&\quad +\left(\sum_{i=r}^b i\right)(a b^{n+1}-c)+(b-r+1)(q+p-1)(a b^{n+2}-c)\\ 
&\quad +\left(\sum_{i=0}^{r-1}i\right)(a b^{n+1}-c)+r(q+p)(a b^{n+2}-c)\\
&=(q-p)\left(p(b+1)^2+\frac{b(b+1)}{2}\right)(a b^{n+1}-c)\\
&\qquad +(b+1)\frac{(q-p-1)(q-p)}{2}(a b^{n+2}-c)\\
&\quad +\left(p(b+1)r+\frac{(r-1)r}{2}\right)(a b^{n+1}-c)+r(q-p)(a b^{n+2}-c)\\
&\quad +\left(\frac{(p-1)p}{2}(b+1)(b+1-r)+p\left(\frac{b(b+1)}{2}-\frac{(r-1)r}{2}\right)\right)(a b^{n+1}-c)\\
&\quad +(b-r+1)\left(\sum_{l=1}^p(q-p+2 l-1)\right)(a b^{n+2}-c)\\ 
&\quad +\left(\frac{(p-1)p}{2}(b+1)r+p\frac{(r-1)r}{2}\right)(a b^{n+1}-c)\\
&\qquad +r\left(\sum_{l=1}^p(q-p+2 l)\right)(a b^{n+2}-c)\\ 
&=\frac{b(b+1)q+r(r-1)-p^2(b+1)^2+p(b+1)\bigl((b+1)(2 q-1)+2 r\bigr)}{2}\\
&\qquad \times(a b^{n+1}-c)\\
&\quad+\left(\frac{(b+1)\bigl(q(q-1)+p(p+1)\bigr)}{2}+q r\right)(a b^{n+2}-c)\\
&=\frac{a b^n-c}{2}\bigl((a b^n-c-1)(a b^{n+1}-c)-b q(2(a b^n-c)-(b+1)(q+1))\\
&\qquad +(b+1)p(2(a b^{n+1}-c)-b(p+1))\bigr)\,.
\end{align*}
Hence, by Lemma \ref{lem:p-sylv}, we obtain that 
\begin{align*} 
&n_p(A_{3,n})\\
&=\frac{1}{2}\bigl((a b^n-c-1)(a b^{n+1}-c-1)-b q(2(a b^n-c)-(b+1)(q+1))\\
&\qquad +(b+1)p(2(a b^{n+1}-c)-b(p+1))\bigr)\,.
\end{align*}
\end{proof}

\section{The case $c<0$}  

When $c<0$ and $r>1$, all the frameworks in Tables \ref{tb:g0system}, \ref{tb:g01system}, \ref{tb:g01psystem} and \ref{tb:g012psystem} are the same as the case $c>0$. However, the maximal element in ${\rm Ap}_p(A_{3,n})$ may be different. In fact, when $c<0$ and $r=1$, the relation (\ref{eq:ape0,c>0}) is not valid.  For clarity, we put $c_0=-c>0$.  

In ${\rm Ap}_0(A_{3,n})$, the largest element is at $(r-1,q)$ or at $(b,q-1)$. Since by $r\ge 1$ and $c_0>0$, the relation 
$$ 
(a b^{n+1}+c_0)(r-1)+(a b^{n+2}+c_0)q>(a b^{n+1}+c_0)b+(a b^{n+2}+c_0)(q-1) 
$$ 
is equivalent to the relation 
$$
(r-1)a b^{n+1}>(b-r)c_0
$$ 
Hence, by Lemma \ref{lem:p-frob}, we have 
\begin{align*}
&g_0(A_{3,n})\\
&=\begin{cases}
(a b^{n+1}+c_0)(r-1)+(a b^{n+2}+c_0)q-(a b^n+c_0)&\text{if $(r-1)a b^{n+1}\ge(b-r)c_0$};\\
(a b^{n+1}+c_0)b+(a b^{n+2}+c_0)(q-1)-(a b^n+c_0)&\text{otherwise}\,. 
\end{cases}
\end{align*}   
In ${\rm Ap}_1(A_{3,n})$, we compare the elements at $(r-1,q+1)$, at $(b,q)$, at $(b+r,q-1)$ and at $(2 b+1,q-2)$ under the condition $c_0>0$.  
Since 
\begin{align*}  
(r-1)a b^{n+1}>(b-r)c_0&\Leftrightarrow
\hbox{element at~}(r-1,q+1)>\hbox{element at~}(b,q)\\ 
&\Leftrightarrow \hbox{element at~}(b+r,q-1)>\hbox{element at~}(2 b+1,q-2)\,,\\ 
a b^{n+1}>c_0&\Leftrightarrow
\hbox{element at~}(r-1,q+1)>\hbox{element at~}(b+r,q-1)\\ 
&\Leftrightarrow  
\hbox{element at~}(b,q)>\hbox{element at~}(2 b+1,q-2)\,,
\end{align*}
we have the following.  
\begin{align*}
&g_1(A_{3,n})\\
&=\left\{
\begin{alignedat}{2}
& (a b^{n+1}+c_0)(r-1)+(a b^{n+2}+c_0)(q+1)-(a b^n+c_0)\\
&\qquad\qquad\qquad\qquad\text{if $(r-1)a b^{n+1}\ge c_0\max\{b-r,r-1\}$},\\
& (a b^{n+1}+c_0)b+(a b^{n+2}+c_0)q-(a b^n+c_0)\\
&\qquad\qquad\qquad\qquad\text{if $(b-r)a b^{n+1}\ge(b-r)c_0>(r-1)a b^{n+1}$},\\
& (a b^{n+1}+c_0)(b+r)+(a b^{n+2}+c_0)(q-1)-(a b^n+c_0)\\
&\qquad\qquad\qquad\qquad\text{if $(r-1)a b^{n+1}\ge(b-r)c_0>(b-r)a b^{n+1}$},\\
& (a b^{n+1}+c_0)(2 b+1)+(a b^{n+2}+c_0)(q-2)-(a b^n+c_0)\\
&\qquad\qquad\qquad\qquad\text{if $(b-r)c_0>a b^{n+1}\max\{b-r,r-1\}$}\,.
\end{alignedat}
\right. 
\end{align*}   
In ${\rm Ap}_2(A_{3,n})$, since 
\begin{align*}  
&(r-1)a b^{n+1}>(b-r)c_0\\
&\Leftrightarrow \hbox{element at~}(r-1,q+2)>\hbox{element at~}(b,q+1)\\ 
&\Leftrightarrow \hbox{element at~}(b+r,q)>\hbox{element at~}(2 b+1,q-1)\\ 
&\Leftrightarrow \hbox{element at~}(2 b+r+1,q-2)>\hbox{element at~}(3 b+2,q-3)\,,\\ 
&a b^{n+1}>c_0\\
&\Leftrightarrow \hbox{element at~}(r-1,q+2)>\hbox{element at~}(b+r,q)\\ 
&\qquad >\hbox{element at~}(2 b+r+1,q-2)\\ 
&\Leftrightarrow
\hbox{element at~}(b,q+1)>\hbox{element at~}(2 b+1,q-1)\\ 
&\qquad >\hbox{element at~}(3 b+2,q-3)\,, 
\end{align*}
we have the following. 
\begin{align*}
&g_2(A_{3,n})\\
&=\left\{
\begin{alignedat}{2}
& (a b^{n+1}+c_0)(r-1)+(a b^{n+2}+c_0)(q+2)-(a b^n+c_0)\\
&\qquad\qquad\qquad\qquad\text{if $(r-1)a b^{n+1}\ge c_0\max\{b-r,r-1\}$},\\
& (a b^{n+1}+c_0)b+(a b^{n+2}+c_0)(q+1)-(a b^n+c_0)\\
&\qquad\qquad\qquad\qquad\text{if $(b-r)a b^{n+1}\ge(b-r)c_0>(r-1)a b^{n+1}$},\\
& (a b^{n+1}+c_0)(2 b+r+1)+(a b^{n+2}+c_0)(q-2)-(a b^n+c_0)\\
&\qquad\qquad\qquad\qquad\text{if $(r-1)a b^{n+1}\ge(b-r)c_0>(b-r)a b^{n+1}$},\\
& (a b^{n+1}+c_0)(3 b+2)+(a b^{n+2}+c_0)(q-3)-(a b^n+c_0)\\
&\qquad\qquad\qquad\qquad\text{if $(b-r)c_0>a b^{n+1}\max\{b-r,r-1\}$}\,.
\end{alignedat}
\right. 
\end{align*}

In general, in ${\rm Ap}_p(A_{3,n})$ ($0\le p\le q$), we find that 
\begin{align*}  
&(r-1)a b^{n+1}>(b-r)c_0\\
&\Leftrightarrow \hbox{element at~}(r-1,q+p)>\hbox{element at~}(b,q+p-1)\\ 
&\Leftrightarrow \hbox{element at~}(b+r,q+p-2)>\hbox{element at~}(2 b+1,q+p-3)\\ 
&\Leftrightarrow \hbox{element at~}(2 b+r+1,q+p-4)>\hbox{element at~}(3 b+2,q+p-5)\\
&\Leftrightarrow \cdots\\
&\Leftrightarrow \hbox{element at~}((p-1)b+r+(p-2),q-p+2)>\hbox{element at~}(p b+p-1,q-p+1)\\ 
&\Leftrightarrow \hbox{element at~}(p b+r+p-1,q-p)>\hbox{element at~}((p+1)b+p,q-p-1)\,,\\ 
&a b^{n+1}>c_0\\
&\Leftrightarrow \hbox{element at~}(r-1,q+p)>\hbox{element at~}(b+r,q+p-2)\\ 
&\qquad >\hbox{element at~}(2 b+r+1,q+p-4)>\cdots\\ 
&\qquad >\hbox{element at~}((p-1)b+r+(p-2),q-p+2)\\
&\qquad >\hbox{element at~}(p b+r+p-1,q-p)\\ 
&\Leftrightarrow 
\hbox{element at~}(b,q+p-1)>\hbox{element at~}(2 b+1,q+p-3)\\ 
&\qquad >\hbox{element at~}(3 b+2,q+p-5)>\cdots\\
&\qquad >\hbox{element at~}(p b+p-1,q-p+1)\\
&\qquad >\hbox{element at~}((p+1)b+p,q-p-1)\,.
\end{align*}
Hence, we have the following.  

\begin{theorem} 
Suppose that $\gcd(A_{3,n})=1$, $c_0>0$ and $0\le p\le q$. 
Then 
\begin{align*}
&g_p(A_{3,n})\\
&=\left\{
\begin{alignedat}{2}
& (a b^{n+1}+c_0)(r-1)+(a b^{n+2}+c_0)(q+p)-(a b^n+c_0)\\
&\qquad\qquad\qquad\qquad\text{if $(r-1)a b^{n+1}\ge c_0\max\{b-r,r-1\}$},\\
& (a b^{n+1}+c_0)b+(a b^{n+2}+c_0)(q+p-1)-(a b^n+c_0)\\
&\qquad\qquad\qquad\qquad\text{if $(b-r)a b^{n+1}\ge(b-r)c_0>(r-1)a b^{n+1}$},\\
& (a b^{n+1}+c_0)(p b+r+p-1)+(a b^{n+2}+c_0)(q-p)-(a b^n+c_0)\\
&\qquad\qquad\qquad\qquad\text{if $(r-1)a b^{n+1}\ge(b-r)c_0>(b-r)a b^{n+1}$},\\
& (a b^{n+1}+c_0)((p+1)b+p)+(a b^{n+2}+c_0)(q-p-1)-(a b^n+c_0)\\
&\qquad\qquad\qquad\qquad\text{if $(b-r)c_0>a b^{n+1}\max\{b-r,r-1\}$}\,.
\end{alignedat}
\right. 
\end{align*} 
\label{th:c<0}
\end{theorem}

\section{Example}  

When $a=5$, $b=2$ and $c=19$, we see that $q=7$ and $r=0$. We know that 
$\gcd(5\cdot 2^3-19,5\cdot 2^4-19,5\cdot 2^5-19)=\gcd(21,61,141)=1$.  
Then, by Theorem \ref{th:p-frob}, for $0\le p\le 7$, we have 
\begin{align*} 
g_p(21,61,141)&=2\cdot 61+(7+p-1)\cdot 141-21\\
&=141 p+947\,. 
\end{align*} 
By Theorem \ref{th:p-sylv}, for $0\le p\le 7$, we have 
\begin{align*} 
n_p(21,61,141)&=\frac{1}{2}\bigl((21-1)(61-1)-2\cdot 7(2\cdot 21-(2+1)(7+1))\\
&\qquad +(2+1)p(2\cdot 61-2(p+1))\bigr)\\
&=3(158+60 p-p^2)\,.
\end{align*} 

Set $a=b+1$, $c=1$ and $n=1$. Since 
\begin{align*}
(b+1)b^4-1&=(b^3-b-1)\bigl((b+1)b-1\bigr)+2\bigl((b+1)b^2-1\bigr)\,,\\ 
(b+1)b^5-1&=(b^4-b^3-2 b^2+1)\bigl((b+1)b-1\bigr)+(b^2+b-2)\bigl((b+1)b^2-1\bigr)\\
&\qquad +2\bigl((b+1)b^3-1\bigr)\,,\\ 
\dots\,,&
\end{align*}
we get 
$$
(b+1)b^4-1,(b+1)b^5-1,\dots\in\ang{(b+1)b-1, (b+1)b^2-1, (b+1)b^3-1}\,.
$$ 
Hence, 
our results (finite sequence) in Theorem \ref{th:p-frob} and Theorem \ref{th:p-sylv} for $p=0$ reduce to those of Thabit numerical semigroups of the first kind in base $b$ (infinite sequence) in \cite[Theorem 5.5, Theorem 5.13]{Song20}. However, when $n=2$, since  
\begin{align*}
&(b+1)b^5-1\not\in\ang{(b+1)b^2-1, (b+1)b^3-1,(b+1)b^4-1}\,,\\
&(b+1)b^6-1,(b+1)b^7-1,\dots\in\\
&\qquad\qquad \ang{(b+1)b^2-1, (b+1)b^3-1,(b+1)b^4-1,(b+1)b^5-1}\,,
\end{align*}
our results and those in \cite{Song20} may be different. We do need the results for four variables to extend the results in \cite{Song20}. 

Set $a=b+1$ and $c=-1$, which reduces the Thabit numbers of the second kind in base $b$ in \cite{Song20}. 
When $n=1$, by  
$$
(b+1)b^4+1,(b+1)b^5+1,\dots\in\ang{(b+1)b+1, (b+1)b^2+1, (b+1)b^3+1}\,,
$$ 
the infinite sequence is reduced to the three variables' finite sequence.  However, when $n=2$, the infinite sequence is reduced to four variables. 
If $b=3$, we see that $q=3^n$ and $r=1$. Then, by the second condition in Theorem \ref{th:c<0}, we have 
\begin{align*}
g_p(A_{3,n})&=\bigl(4\cdot 3^{n+1}+1\bigr)\cdot 3+\bigl(4\cdot 3^{n+2}+1\bigr)(3^n+p-1)-\bigl(4\cdot 3^n+1\bigr)\\
&=(4\cdot 3^{n+2}+1)p+4\cdot 3^{2 n+2}-3^{n+1}+1\,. 
\end{align*} 
In particular, when $n=1$, we get 
$$
g_0(A_{3,1})=109 p+316\,. 
$$

\section{Four variables}  

Since it is important to consider general $p$ representations for an infinite sequence in the numerical semigroup, it is natural to extend the argument about $3$ variables to the case of a finite sequence such as $4$ variables, $5$ variables and so on. However, geometrical interpretation, even in the case of four variables, is considerably more complicated than in the case of three variables. One of the reasons is that the elements of the Ap\'ery sets ${\rm Ap}_p(A)$ and ${\rm Ap}_{p+1}(A)$ overlap. Here, only a sketch shall be described. 

For given positive integers $a$, $b$ and $c$ with $b\ge 2$, nonnegative integers $\alpha$, $\beta$ and $\gamma$ are determined as 
\begin{align}
&a b^n-c=\alpha(b^2+b+1)+\beta(b+1)+\gamma,\notag\\ 
&0\le\beta(b+1)+\gamma\le b^2+b,\quad 0\le\gamma\le b\,. 
\label{abg}
\end{align}
That is, 
$$
\alpha=\fl{\frac{a b^n-c}{b^2+b+1}}\quad\hbox{and}\quad\beta=\frac{a b^n-c-\beta(b^2+b+1)}{b+1}\,. 
$$ 
Then the elements in ${\rm Ap}_0(A_{4,n})$, where $A_{4,n}=\{a b^n-c, a b^{n+1}-c, a b^{n+2}-c, a b^{n+3}-c\}$ are distributed into $\alpha$ blocks consisting of $b^2+b+1$ elements and one incomplete block. The incomplete block consists of $\beta$ rows of $b+1$ elements and one row of incomplete $\gamma$ elements.  
For simplicity, write $(x_2,x_3,x_4)$ to represent the corresponding element $x_2(a b^{n+1}-c)+x_3(a b^{n+2}-c)+x_4(a b^{n+3}-c)$. Note that the leftmost column of the upper block and the rightmost column of the lower block are in one column, and the elements are duplicated in one column of the lower block. For example, one element can be expressed in two ways: $(0,b+j+1,0)$ and $(b,j,1)$ ($0\le j\le b-1$). 
By comparing the possible candidates to take the maximal value, we find that the element at $(\gamma-1,\beta,\alpha)$ takes the largest in ${\rm Ap}_0(A_{4,n})$ when $\gamma\ge 1$. When $\gamma=0$, the element at $(b,\beta-1,\alpha)$ or $(0,b+\beta,\gamma-1)$ (overlapped same value) takes the largest.

\begin{table}[htbp]
  \centering
\scalebox{0.5}{
\begin{tabular}{ccccccccccccc}
\cline{11-12}\cline{13-13}
&&&&&&&&&&\multicolumn{1}{|c}{$(0,0,0)$}&$\cdots$&\multicolumn{1}{c|}{$(b,0,0)$}\\
&&&&&&&&&&\multicolumn{1}{|c}{$\vdots$}&&\multicolumn{1}{c|}{$\vdots$}\\
&&&&&&&&&&\multicolumn{1}{|c}{$(0,b-1,0)$}&$\cdots$&\multicolumn{1}{c|}{$(b,b-1,0)$}\\
\cline{12-13}
&&&&&&&&&&\multicolumn{1}{|c|}{$(0,b,0)$}&&\\
\cline{9-10}\cline{11-11}
&&&&&&&&\multicolumn{1}{|c}{$(0,0,1)$}&$\cdots$&\multicolumn{1}{c|}{$(0,b+1,0)\atop(b,0,1)$}&&\\
&&&&&&&&\multicolumn{1}{|c}{$\vdots$}&&\multicolumn{1}{c|}{$\vdots$}&&\\
&&&&&&&&\multicolumn{1}{|c}{$(0,b-1,1)$}&$\cdots$&\multicolumn{1}{c|}{$(0,2 b,0)\atop(b,b-1,1)$}&&\\
\cline{10-11}
&&&&&&&&\multicolumn{1}{|c|}{$(0,b,1)$}&&&&\\
\cline{7-8}\cline{9-9}
&&&&&&\multicolumn{1}{|c}{$(0,0,2)$}&$\cdots$&\multicolumn{1}{c|}{$(0,b+1,1)\atop(b,0,2)$}&&&&\\
&&&&&&\multicolumn{1}{|c}{$\vdots$}&&\multicolumn{1}{c|}{$\vdots$}&&&&\\
&&&&&&\multicolumn{1}{|c}{$(0,b-1,2)$}&$\cdots$&\multicolumn{1}{c|}{$(0,2 b,1)\atop(b,b-1,2)$}&&&&\\
\cline{8-9}
&&&&&&\multicolumn{1}{|c|}{$(0,b,2)$}&&&&&&\\
\cline{7-7}
&&&&&$\iddots$&&&&&&&\\
\cline{3-4}\cline{5-5}
&&\multicolumn{1}{|c}{$(0,0,\alpha-1)$}&$\cdots$&\multicolumn{1}{c|}{$(0,b+1,\alpha-2)\atop(b,0,\alpha-1)$}&&&&&&&&\\
&&\multicolumn{1}{|c}{$\vdots$}&&\multicolumn{1}{c|}{$\vdots$}&&&&&&&&\\
&&\multicolumn{1}{|c}{$(0,b-1,\alpha-1)$}&$\cdots$&\multicolumn{1}{c|}{$(0,2 b,\alpha-2)\atop(b,b-1,\alpha-1)$}&&&&&&&&\\
\cline{4-5}
&&\multicolumn{1}{|c|}{$(0,b,\alpha-1)$}&&&&&&&&&&\\
\cline{1-2}\cline{3-3}
\multicolumn{1}{|c}{$(0,0,\alpha)$}&$\cdots$&\multicolumn{1}{c|}{$(0,b+1,\alpha-1)\atop(b,0,\alpha)$}&&&&&&&&&&\\
\multicolumn{1}{|c}{$\vdots$}&&\multicolumn{1}{c|}{$\vdots$}&&&&&&&&&&\\
\multicolumn{1}{|c}{$(0,\beta-1,\alpha)$}&$\cdots$&\multicolumn{1}{c|}{$(0,b+\beta,\alpha-1)\atop(b,\beta-1,\alpha)$}&&&&&&&&&&\\
\cline{3-3}
\multicolumn{1}{|c}{$(0,\beta,\alpha)\cdots$}&\multicolumn{1}{c|}{$(\gamma-1,\beta,\alpha)$}&&&&&&&&&&&\\
\cline{1-2}
\end{tabular}
} 
  \caption{Complete residue system ${\rm Ap}_0(A_{4,n})$}
  \label{tb:g4_0system}
\end{table}

The element of ${\rm Ap}_1(A_{4,n})$ when $p=1$ is continuously arranged so as to fill the gap of the upper right block with respect to the element of the same residue modulo $a b^n-c$ of ${\rm Ap}_0(A_{4,n})$. Only the elements of the top block will generate a new block just to the right of it, going up one line. However, only the elements in the top row are placed in succession to fill the gaps between the elements in the bottom left incomplete block. This is similar to that in the case of three variables. 

The decisive difference from the case of three variables is that the overlapping element part does not move this time but moves at the next $p+1$, in this case when $p=1$. Also, when the elements coming from the top row fill the incomplete block at the bottom, that is, when the number of elements reaches $b^2+b+1$, then it starts to be arranged so as to form a new block at the lower left. 

The maximal value in ${\rm Ap}_1(A_{4,n})$ comes just below that in ${\rm Ap}_0(A_{4,n})$.   

When $p=2,3,\dots$, the similar procedure is continued. However, as there are only $b$ rows except the first column (where is $b+1$ rows) in each block, the regularity of the position to take the maximal value breaks down for $p>b-\beta$.

\begin{theorem}  
Let $\alpha$, $\beta$ and $\gamma$ be nonnegative integers satisfying the condition (\ref{abg}). Then for $0\le p\le b-\beta$ 
\begin{align*}
&g_p(A_{4,n})\\ 
&=\left\{
\begin{alignedat}{2}
& (\gamma-1)(a b^{n+1}-c)+(\beta+p)(a b^{n+2}-c)+\alpha(a b^{n+3}-c)-(a b^{n}-c)\\
&\qquad\qquad\qquad\qquad\qquad\text{if $\gamma\ge 1$},\\
& (b+\beta+p)(a b^{n+2}-c)+(\alpha-1)(a b^{n+3}-c)-(a b^{n}-c)\\
&\qquad\qquad\qquad\qquad\qquad\text{if $\gamma=0$}\,.
\end{alignedat}
\right. 
\end{align*} 
\label{th:4variable}
\end{theorem}

We illustrate the case where $a=2$, $b=3$, $c=37$ and $n=3$. Note that $\gcd(2\cdot 3^3-37, 2\cdot 3^4-37, 2\cdot 3^5-37, 2\cdot 3^6-37)=\gcd(17,125,449,1421)=1$. Since $17=1\cdot(3^2+3+1)+1\cdot(3+1)+0$, we know that $\alpha=\beta=1$ and $\gamma=0$. The values of elements $j\pmod{17}$ appear in the table. We see that $c'=(b-1)c=74\equiv 6\pmod{17}$. 
In Table, $\ctext{j}$ ($j=0,1,2$) shows the range in which the elements in ${\rm Ap}_j(A_{4,n})$ exist. 

\begin{table}[htbp]
  \centering
\scalebox{0.4}{
\begin{tabular}{ccccccccccccccccccc}
\cline{7-8}\cline{9-10}\cline{11-12}\cline{13-14}\cline{15-16}\cline{17-18}\cline{19-19}
&&&&&&\multicolumn{1}{|c}{$(0,0,0)$}&$(1,0,0)$&$(2,0,0)$&\multicolumn{1}{c|}{$(3,0,0)$}&$(4,0,0)$&$(5,0,0)$&$(6,0,0)$&\multicolumn{1}{c|}{$(7,0,0)$}&$(8,0,0)$&$(9,0,0)$&$(10,0,0)$&\multicolumn{1}{c|}{$(11,0,0)$}&\multicolumn{1}{c|}{$(12,0,0)$}\\
\cline{16-17}\cline{18-19}
&&&&&&\multicolumn{1}{|c}{$(0,1,0)$}&$(1,1,0)$&$(2,1,0)$&\multicolumn{1}{c|}{$(3,1,0)$}&$(4,1,0)$&$(5,1,0)$&$(6,1,0)$&\multicolumn{1}{c|}{$(7,1,0)$}&\multicolumn{1}{c|}{$(8,1,0)$}&$(9,1,0)$&$(10,1,0)$&\multicolumn{1}{c|}{$(11,1,0)$}&\\
\cline{12-12}\cline{13-14}\cline{15-16}\cline{17-18}
&&&&&&\multicolumn{1}{|c}{$(0,2,0)$}&$(1,2,0)$&$(2,2,0)$&\multicolumn{1}{c|}{$(3,2,0)$}&\multicolumn{1}{c|}{$(4,2,0)$}&$(5,2,0)$&$(6,2,0)$&\multicolumn{1}{c|}{$(7,2,0)$}&\multicolumn{1}{c|}{$(8,2,0)$}&&&&\\
\cline{8-9}\cline{10-11}\cline{12-13}\cline{14-15}
&&&&&&\multicolumn{1}{|c|}{$(0,3,0)$}&$(1,3,0)$&$(2,3,0)$&\multicolumn{1}{c|}{$(3,3,0)$}&\multicolumn{1}{c|}{$(4,3,0)$}&&&&&&&&\\
\cline{4-5}\cline{6-7}\cline{8-9}\cline{10-11}
&&&\multicolumn{1}{|c}{$(0,0,1)$}&$(1,0,1)$&$(2,0,1)$&\multicolumn{1}{|c|}{$(3,0,1)\atop(0,4,0)$}&$(4,0,1)\atop(1,4,0)$&$(5,0,1)\atop(2,4,0)$&\multicolumn{1}{c|}{$(6,0,1)\atop(3,4,0)$}&\multicolumn{1}{c|}{$(7,0,1)\atop(4,4,0)$}&&&&&&&&\\
\cline{4-5}\cline{6-7}\cline{8-9}\cline{10-11}
&&&\multicolumn{1}{|c}{$(0,1,1)$}&$(1,1,1)$&$(2,1,1)$&\multicolumn{1}{|c|}{$(3,1,1)\atop(0,5,0)$}&$(4,1,1)\atop(1,5,0)$&$(5,1,1)\atop(2,5,0)$&\multicolumn{1}{c|}{$(6,1,1)\atop(3,5,0)$}&&&&&&&&&\\
\cline{4-5}\cline{6-7}\cline{8-9}\cline{10-10}
&&&\multicolumn{1}{|c}{$(0,2,1)$}&$(1,2,1)$&$(2,2,1)$&\multicolumn{1}{|c|}{$(3,2,1)\atop(0,6,0)$}&&&&&&&&&&&&\\
\cline{4-5}\cline{6-7}
&&&\multicolumn{1}{|c|}{$(0,3,1)$}&&&&&&&&&&&&&&&\\
\cline{1-2}\cline{3-4}
\multicolumn{1}{|c}{$(0,0,2)$}&$(1,0,2)$&\multicolumn{1}{c|}{$(2,0,2)$}&&&&&&&&&&&&&&&&\\
\cline{1-2}\cline{3-3}
\end{tabular}
}  
\vskip1cm 
\scalebox{0.6}{
\begin{tabular}{ccccccccccccccccccc}
\cline{7-8}\cline{9-10}\cline{11-12}\cline{13-14}\cline{15-16}\cline{17-18}\cline{19-19}
&&&&&&\multicolumn{1}{|c}{$0$}&$6 c'$&$12 c'$&\multicolumn{1}{c|}{$c'$}&$7 c'$&$13 c'$&$2 c'$&\multicolumn{1}{c|}{$8 c'$}&$14 c'$&$3 c'$&$9 c'$&\multicolumn{1}{c|}{$15 c'$}&\multicolumn{1}{c|}{$4 c'$}\\
\cline{16-17}\cline{18-19}
&&&&&&\multicolumn{1}{|c}{$7 c'$}&$13 c'$&$2 c'$&\multicolumn{1}{c|}{$8 c'$}&$14 c'$&$3 c'$&$9 c'$&\multicolumn{1}{c|}{$15 c'$}&\multicolumn{1}{c|}{$4 c'$}&$10 c'$&$16 c'$&\multicolumn{1}{c|}{$5 c'$}&\\
\cline{12-12}\cline{13-14}\cline{15-16}\cline{17-18}
&&&&&&\multicolumn{1}{|c}{$14 c'$}&$3 c'$&$9 c'$&\multicolumn{1}{c|}{$15 c'$}&\multicolumn{1}{c|}{$4 c'$}&$10 c'$&$16 c'$&\multicolumn{1}{c|}{$5 c'$}&\multicolumn{1}{c|}{$11 c'$}&&&&\\
\cline{8-9}\cline{10-11}\cline{12-13}\cline{14-15}
&&&&&&\multicolumn{1}{|c|}{$4 c'$}&$10 c'$&$16 c'$&\multicolumn{1}{c|}{$5 c'$}&\multicolumn{1}{c|}{$11 c'$}&&&&&&&&\\
\cline{4-5}\cline{6-7}\cline{8-9}\cline{10-11}
&&&\multicolumn{1}{|c}{$10 c'$}&$16 c'$&$5 c'$&\multicolumn{1}{|c|}{$11 c'$}&$0$&$6 c'$&\multicolumn{1}{c|}{$12 c'$}&\multicolumn{1}{c|}{$c'$}&&&&&&&&\\
\cline{4-5}\cline{6-7}\cline{8-9}\cline{10-11}
&&&\multicolumn{1}{|c}{$0$}&$6 c'$&$12 c'$&\multicolumn{1}{|c|}{$c'$}&$7 c'$&$13 c'$&\multicolumn{1}{c|}{$2 c'$}&&&&&&&&&\\
\cline{4-5}\cline{6-7}\cline{8-9}\cline{10-10}
&&&\multicolumn{1}{|c}{$7 c'$}&$13 c'$&$2 c'$&\multicolumn{1}{|c|}{$8 c'$}&&&&&&&&&&&&\\
\cline{4-5}\cline{6-7}
&&&\multicolumn{1}{|c|}{$14 c'$}&&&&&&&&&&&&&&&\\
\cline{1-2}\cline{3-4}
\multicolumn{1}{|c}{$3 c'$}&$9 c'$&\multicolumn{1}{c|}{$15 c'$}&&&&&&&&&&&&&&&&\\
\cline{1-2}\cline{3-3}
\end{tabular}
}  
\vskip1cm 
\scalebox{0.65}{
\begin{tabular}{ccccccccccccccccccc}
\cline{7-8}\cline{9-10}\cline{11-12}\cline{13-14}\cline{15-16}\cline{17-18}\cline{19-19}
&&&&&&\multicolumn{1}{|c}{}&&&\multicolumn{1}{c|}{}&&&&\multicolumn{1}{c|}{}&&$\ctext{2}$&&\multicolumn{1}{c|}{}&\multicolumn{1}{c|}{$\ctext{3}$}\\
\cline{16-17}\cline{18-19}
&&&&&&\multicolumn{1}{|c}{}&$\ctext{0}$&&\multicolumn{1}{c|}{}&&$\ctext{1}$&&\multicolumn{1}{c|}{}&\multicolumn{1}{c|}{}&&$\ctext{3}$&\multicolumn{1}{c|}{}&\\
\cline{12-12}\cline{13-14}\cline{15-16}\cline{17-18}
&&&&&&\multicolumn{1}{|c}{}&&&\multicolumn{1}{c|}{}&\multicolumn{1}{c|}{}&&$\ctext{2}$&\multicolumn{1}{c|}{}&\multicolumn{1}{c|}{$\ctext{3}$}&&&&\\
\cline{8-9}\cline{10-11}\cline{12-13}\cline{14-15}
&&&&&&\multicolumn{1}{|c|}{}&&$\ctext{1}$&\multicolumn{1}{c|}{}&\multicolumn{1}{c|}{$\ctext{2}$}&&&&&&&&\\
\cline{4-5}\cline{6-7}\cline{8-9}\cline{10-11}
&&&\multicolumn{1}{|c}{}&$\ctext{0}$&&\multicolumn{1}{|c|}{$\ctext{0}$$\ctext{1}$}&&$\ctext{2}$$\ctext{3}$&\multicolumn{1}{c|}{}&\multicolumn{1}{c|}{$\ctext{3}$$\ctext{4}$}&&&&&&&&\\
\cline{4-5}\cline{6-7}\cline{8-9}\cline{10-11}
&&&\multicolumn{1}{|c}{}&$\ctext{1}$&&\multicolumn{1}{|c|}{$\ctext{1}$$\ctext{2}$}&&$\ctext{3}$$\ctext{4}$&\multicolumn{1}{c|}{}&&&&&&&&&\\
\cline{4-5}\cline{6-7}\cline{8-9}\cline{10-10}
&&&\multicolumn{1}{|c}{}&$\ctext{2}$&&\multicolumn{1}{|c|}{$\ctext{2}$$\ctext{3}$}&&&&&&&&&&&&\\
\cline{4-5}\cline{6-7}
&&&\multicolumn{1}{|c|}{$\ctext{3}$}&&&&&&&&&&&&&&&\\
\cline{1-2}\cline{3-4}
\multicolumn{1}{|c}{}&$\ctext{3}$&\multicolumn{1}{c|}{}&&&&&&&&&&&&&&&&\\
\cline{1-2}\cline{3-3}
\end{tabular}
} 
  \caption{${\rm Ap}_p(17,125,449,1421)$ ($p=0,1,2,3$)}
  \label{tb:g4_17-125-449-1421}
\end{table}

By Theorem \ref{th:4variable}, we have for $0\le p\le 3-1$ 
\begin{align*}
g_p(A_{4,3})&=(3+1+p)\cdot 449+(1-1)\cdot 1421-17\\
&=449 p+1779\,.  
\end{align*}  

However, when $p=3$, no element comes to the position at $(3,3,1)$ (or $(0,7,0)$). The maximal element is at $(2,0,2)$.

%\section*{Acknowledgements}  
%The author thanks the anonymous referee for his/her careful reading of this paper and useful comments.  

\end{document}